\documentclass{amsart}
\usepackage{amsfonts}
\usepackage{amsmath}
\usepackage{amssymb}
\usepackage{graphicx}%
\providecommand{\U}[1]{\protect\rule{.1in}{.1in}}
\providecommand{\U}[1]{\protect\rule{.1in}{.1in}}
\providecommand{\U}[1]{\protect\rule{.1in}{.1in}}
\providecommand{\U}[1]{\protect\rule{.1in}{.1in}}
\providecommand{\U}[1]{\protect\rule{.1in}{.1in}}
\providecommand{\U}[1]{\protect\rule{.1in}{.1in}}

\newtheorem{theorem}{Theorem}[section]

\newtheorem{prop}[theorem]{Proposition}
\newtheorem{lemma}[theorem]{Lemma}
\newtheorem{corollary}[theorem]{Corollary}

\theoremstyle{definition}
\newtheorem{definition}[theorem]{Definition}
\newtheorem{example}[theorem]{Example}

\newtheorem{remark}[theorem]{Remark}

\numberwithin{equation}{section}


\begin{document}
\title[Attractors of iterated function systems]{The Conley Attractor of an Iterated Function System}
\author{Michael F. Barnsley }
\address{Department of Mathematics\\
Australian National University\\
Canberra, ACT, Australia}
\email{michael.barnsley@maths.anu.edu.au, mbarnsley@aol.com}
\urladdr{http://www.superfractals.com}
\author{Andrew Vince}
\address{Department of Mathematics\\
University of Florida\\
Gainesville, FL 32611-8105, USA}
\email{avince@ufl.edu}

\begin{abstract}
We investigate the topological and metric properties of attractors of an
iterated function system (IFS) whose functions may not be contractive. We
focus, in particular, on invertible IFSs of finitely many maps on a compact
metric space. We rely on ideas Kieninger and McGehee and Wiandt, restricted to
what is, in many ways, a simpler setting, but focused on a special type of
attractor, namely point-fibred minimal (locally) invariant sets. This allows
us to give short proofs of some of the key ideas. 

\end{abstract}
\maketitle

\section{Introduction}

The subject of this paper is the attractor or attractors of an iterated
function system (IFS) on a compact metric space. Iterated function systems are
used for the construction of deterministic fractals \cite{B1} and have found
numerous applications, in particular to image compression and image processing
\cite{B2}. The notion of an attactor of an IFS has historically been linked
with the the contractive properties of the functions in the IFS, beginning
with the work of Huchinson \cite{H}. If the functions in the IFS are
contractions, then the existence of an attractor, in a strong sense which we
call a \textit{strict attractor}, is assured. Moreover, it has recently been
shown \cite{ABVW,V} that, for affine and M\"{o}bius IFSs (defined in
Section~\ref{sec:attractor}), the existence of an attractor implies that the
functions in the IFS must be contractions. There do exists, however, examples
of IFS attractors for which the functions are not contractions with respect to
any metric that gives the same topology as the underlying space \cite{BV}. In
the current paper we investigate the topological and metric properties of
attractors of a general IFS on a compact metric space for which the functions
are not necessarily contractive. We rely on ideas in \cite{kieninger} and
\cite{mcgehee2}, but restricted to what is, in many ways, a simpler setting.

There are numerous definitions of an attractor; see \cite{milnor,milnor2} and
\cite{HH} for example. The notion of attractor as used in the paragraph above
is the one that has become standard in the fractal geometry literature. For
this paper the term \textit{strict attractor} is used for this type of
attractor, because it is natural in our setting to introduce a more general
notion of attractor. For this type of attractor the term \textit{Conley
attractor} is used because it is essentially an extension to a finite set of
functions of a notion used so successfully by Conley for a single function
\cite{C1}. Both the strict attractor and the Conley attractor are defined in
Section~\ref{sec:attractor}. The attractor block, an essential tool for our
investigation of Conley attractors, is the subject of Section~\ref{sec:block},
the main result being Theorem~\ref{thm:block} stating that every Conley
attractor possesses an attroctor block.

Section~\ref{sec:SC} gives a couple of sufficient conditions that guarantee
that a Conley attractor is a strict attractor. These sufficient conditions
involve both contractive properties of the functions in the IFS and the
existence of a natural addressing function for the points of the attractor.
The notion of ``fibering" plays a role; the thesis of Kieninger
\cite{kieninger} has an extensive discussion of the subject of fibering.

In the case that the functions in the IFS $\mathcal{F}$ are invertible, there
is a duality between the action of the IFS ${\mathcal{F}}$ and the IFS
${\mathcal{F}}^{\ast}$ consisting of the inverses of the maps in
${\mathcal{F}}$. This leads to the notion of an attractor-repller pair
$(A,A^{\ast})$ consisting of a Conley attractor $A$ of ${\mathcal{F}}$ and a
disjoint Conley attractor $A^{\ast}$ of ${\mathcal{F}}^{\ast}$, the main
result being Theorem~\ref{thm:ARpair} in Section~\ref{sec:AR}. For the
dynamics of a single function, this plays a significant role in Conley's index
theory \cite{C1} and has been extended to the context of \textquotedblleft
closed relations" by McGehee and Wiandt \cite{mcgehee1, mcgehee2}. In general,
an invertible IFS can have many Conley attractor-repeller pairs. The second
main result in Section~\ref{sec:AR} is Theorem~\ref{cmwthm}, which relates the
structure of these Conley attractor-repeller pairs to the dynamics of the IFS
$F$, more specifically to the set of chain-recurrent points of $F$.

The last section provides some examples of the properties described in the paper.

\section{Attractors}

\label{sec:attractor}

Unless otherwise stated, throughout this section $(\mathbb{X},d)$ is a
complete metric space. The closure of a set $B$ is denoted $\overline B$ and
the interior by $B^{o}$.

\begin{definition}
If $f_{n}:\mathbb{X}\rightarrow\mathbb{X}$, $n=1,2,\dots,N,$ are continuous
functions, then $\mathcal{F}=\left(  \mathbb{X};f_{1},f_{2},...,f_{N}\right)
$ is called an \textbf{iterated function system} (IFS). If each of the maps
$f\in\mathcal{F}$ is a homeomorphism then $\mathcal{F}$ is said to be
\textbf{invertible}, and the notation $\mathcal{F}^{*}:=\left(  \mathbb{X}%
;f_{1}^{-1},f_{2}^{-1},...,f_{N}^{-1}\right)  $ is used.

Subsequently in this paper we refer to some special cases of an IFS. For an
\textbf{affine IFS} we have $\mathbb{X }= {\mathbb{R}}^{n}$ and the functions
in the IFS are affine functions of the form $f(x) = Ax+a$, where $A$ is an
$n\times n$ matrix and $a \in{\mathbb{R}}^{n}$. For a \textbf{projective IFS}
we have $\mathbb{X }= \mathbb{RP}^{n}$, real projective space, and the
functions in the IFS are projective functions of the form $f(x) = Ax$, where
$A$ is an $(n+1)\times(n+1)$ matrix and $x$ is given by homogeneous
coordinates. For a \textbf{M\"obius IFS} we have $\mathbb{X }=
\widehat{\mathbb{C}}= \mathbb{C }\cup\{\infty\}$, the extended complex plane,
and the functions are M\"obius functions of the form $f(z) =\frac{az+b}{cd+d}%
$, where $ad=bc = 1$. M\"obius functions may equivalently be considered as
acting on the Riemann sphere or the complex projective line.
\end{definition}

By a slight abuse of terminology we use the same symbol $\mathcal{F}$ for the
IFS, the set of functions in the IFS, and for the following mapping. Letting
$2^{\mathbb{X}}$ denote the collection of subsets of $\mathbb{X}$, define
$\mathcal{F}:2^{\mathbb{X}}\mathbb{\rightarrow}2^{\mathbb{X}}$ by
\[
\mathcal{F}(B)=\bigcup_{f\in\mathcal{F}}f(B)
\]
for all $B\in2^{\mathbb{X}}$. Let $\mathbb{H=H(X)}$ be the set of nonempty
compact subsets of $\mathbb{X}$. Since $\mathcal{F}\left(  \mathbb{H}\right)
\subseteq\mathbb{H}$ we can also treat $\mathcal{F}$ as a mapping
$\mathcal{F}:\mathbb{H\rightarrow H}$. Let $d_{\mathbb{H}}$ denote the
Hausdorff metric on $\mathbb{H}$, which can be defined as follows. Using the
notation
\[
S_{r}=\{y\in\mathbb{X}\,:\,d_{\mathbb{X}}(x,y)<r\text{ for some }x\in S\}
\]
with $S\subset\mathbb{X}$ and $r>0,$ a convenient definition of the Hausdorff
metric $d_{\mathbb{H}}$ (see for example \cite[p.66]{edgar}) is
\[
d_{\mathbb{H}}(B,C)=\inf\{r>0:B\subset C_{r}\text{ and }C\subset B_{r}\}
\]
for all $B,C\in\mathbb{H}$. Under various conditions $\mathcal{F}%
:\mathbb{H\rightarrow H}$ is continuous with respect to $d_{\mathbb{H}}$. This
occurs for example when the metric space $\mathbb{X}$ is compact or when each
$f\in\mathcal{F}$ is Lipshitz, see \cite{barnvince1}. It was also proved to be
true when $\mathbb{X}$ is a complete metric space, see \cite{BarnLesCty}.

For $B\subset\mathbb{X}$ and $k\in\mathbb{N}:=\{1,2,...\}$, let $\mathcal{F}%
^{k}(B)$ denote the $k$-fold composition of $\mathcal{F}$, the union of
$f_{i_{1}}\circ f_{i_{2}}\circ\cdots\circ f_{i_{k}}(B)$ over all finite words
$i_{1}i_{2}\cdots i_{k}$ of length $k.$ Define $\mathcal{F}^{0}(B)=B.$ The
definition of an attractor that is fairly standard in the literature on
fractals is as follows.

\begin{definition}
\label{def:strictattractor} A non-empty set $A\in\mathbb{H(X)}$ is said to be
a \textbf{strict attractor} of the IFS $\mathcal{F}$ if

(i) $\mathcal{F}(A)=A$ and

(ii) there is an open set $U\subset\mathbb{X}$ such that $A\subset U$ and
$\lim_{k\rightarrow\infty}\mathcal{F}^{k}(S)=A$ for all $S\in\mathbb{H(}U)$,
where the limit is with respect to the Hausdorff metric.

The largest open set $U$ such that (ii) is true is called the \textbf{basin}
of the strict attractor $A$ of the IFS $\mathcal{F}$.
\end{definition}

The following less restrictive (see statement (4) of
Proposition~\ref{prop:attractor}) definition of an attractor $A$ is used in
this paper. Called a Conley attractor, it generalizes a notion of attractor
due to Conley \cite{C1} that has proved useful in the study of the dynamics of
a single function. The definition states, in the case of nonempty attractors,
that there is an open set $U$ containing $A$ whose closure converges, in the
Hausforff sense, to $A$.

\begin{definition}
\label{def:ConleyAttractor} A compact set $A$ is said to be a \textbf{Conley
attractor} of the IFS $\mathcal{F}$ if there is an open set $U$ such that
$A\subset U$ and
\[
A=\lim_{k\rightarrow\infty}\mathcal{F}^{k}({\overline{U}}).
\]
The \textbf{basin} of $A$ is the union of all open sets $U$ that satisfy the
above definition.
\end{definition}

The empty set is always a Conley attractor of an IFS and $\mathbb{X}$ is an
attractor if ${\mathcal{F}}$ contains at least one surjective function.
Example~\ref{ex:multiple} provides an IFS with infinitely many Conley attractors.

\begin{prop}
\label{prop:attractor} Let $\mathcal{F}$ be an IFS on a compact metric space.

\begin{enumerate}
\item If $A$ is a Conley attractor of $\mathcal{F}$, then ${\mathcal{F}}(A) =
A$.

\item If $A$ is a Conley attractor of ${\mathcal{F}}$ with basin $B$ and $S$
is any compact set such that $A \subseteq S \subset B$, then $\lim
_{k\rightarrow\infty}\mathcal{F}^{k}(S) = A.$

\item If $A$ and $A^{\prime}$ are Conley attractors of an IFS $\mathcal{F}$,
then $A\cup A^{\prime}$ and $A\cap A^{\prime}$ are also Conley attractors.

\item If $A$ is a strict attractor with basin $B$, then $A$ is a Conley
attractor with basin $B$.
\end{enumerate}
\end{prop}

\begin{proof}
Concerning statement (1), ${\mathcal{F}}(A) = {\mathcal{F}}(\lim
_{k\rightarrow\infty}{\mathcal{F}}^{k}({\overline U}) )= \lim_{k\rightarrow
\infty}{\mathcal{F}}^{k+1}({\overline U}) = A$.

Statement (2) follows from statement (1) and the fact that $S$, being compact,
is contrained in the union of finitely many open sets $U$ such that $A =
\lim_{k\rightarrow\infty}\mathcal{F}^{k}({\overline U}).$

Concerning statement (3), in Definition~\ref{def:ConleyAttractor}, let $U$ be
the open set for $A$ and $U^{\prime}$ the open set for $A^{\prime}$. Then
$U\cup U^{\prime}$ and $U\cap U^{\prime}$ are the required open sets for
$A\cup A^{\prime}$ and $A\cap A^{\prime}$.

Concerning statement (4), let $U$ be an open set containing $A$ such that
${\overline{U}}\subset B$. Then $U$ satisfies the conditions in the definition
of a Conley attractor. Let $B^{\prime}$ denote the basin of $A$ as given in
Definition~\ref{def:ConleyAttractor}. To show that $B^{\prime}=B$, first note
that $\bigcup\{U\,:\,A\subset U\subset{\overline{U}}\subset B\}=B$. Therefore
$B\subseteq B^{\prime}$. Moreover, if $S$ is any compact subset of $B^{\prime
}$, then there is an open set $U$ such that $A\cup S\subset U\subset
{\overline{U}}\subset B^{\prime}$. Since $\lim_{k\rightarrow\infty
}{\mathcal{F}}^{k}({\overline{U}})=A$, there is a $k$ such that ${\mathcal{F}%
}^{k}(S)\subset B$. It follows that $\lim_{k\rightarrow\infty}{\mathcal{F}%
}^{k}(S)=A$.
\end{proof}

\begin{remark}
Conley's concept of an attractor for one function is usually expressed as an
$\omega$-limit. Although it is slightly more complicated to do so, our
definition of Conley attractor could be defined in a similar manner. Let
$S\subset\mathbb{X}$. The $\omega$\textbf{-limit set} of the set $S$ under
$\mathcal{F}$ is
\[
\omega(S): = \bigcap\limits_{K\geq1}\overline{\bigcup_{k\geq K}\mathcal{F}%
^{k}(S)}\text{.}%
\]
Omitting the proof, we state that a set $A \subset\mathbb{X}$ is a Conley
attractor of the IFS $\mathcal{F}$ if and only if

(i) $A = {\omega}(U)$ for some open subset $U$ of $\mathbb{X}$, and

(ii) $A \subset U$.

\noindent Moreover, the largest open set $B\subset\mathbb{X}$ such that
${\omega}(\{x\}) \subset A$ for all $x \in B$ is the basin of $A$. It follows
from the equivalence of the two definitions that the Hausdorff limit in
Definition~\ref{def:ConleyAttractor} exists if and only if $\overline
{\bigcup_{k\geq K}\mathcal{F}^{k}(U)} \subset U$ for some $K$.
\end{remark}

In the following lemma we use the notation $\vec{d}(X,Y) = \max_{x\in X}
\min_{y \in Y} d(x,y)$ for compact sets $X$ and $Y$. The lemma states that the
basin of a Conley attractor $A$ consists of those points whose image under
iterates of ${\mathcal{F}}$ get arbitrarily close to $A$.

\begin{lemma}
\label{lem:basin} If $A$ is a Conley attractor of an IFS ${\mathcal{F}}$, then
the basin of $A$ is
\begin{equation}
\label{eq:basin}B = \left\{  x \, :\, \lim_{k \rightarrow\infty} \vec
{d}({\mathcal{F}}^{k}(x),A) = 0\right\}  = \left\{ x \, : \, \omega(\{x\})
\subset A\right\} .
\end{equation}

\end{lemma}

\begin{proof}
Let $B$ denote the basin for $A$ and $B^{\prime}$ the set in
Equation~\ref{eq:basin}. It follows from the definitions that $B \subseteq
B^{\prime}$.

Let $U$ be an open set containing $A$ such that $\lim_{k \rightarrow\infty}
{\mathcal{F}}^{k}({\overline U}) = A$. To prove that $B^{\prime}\subseteq B$,
it suffices to show that, for any $x \in B^{\prime}$, there is an open
neighborhood $N$ of $x$ such that, if $U^{\prime}= U\cup N$, then $\lim_{k
\rightarrow\infty} {\mathcal{F}}^{k}({\overline U^{\prime}}) = A$, and hence
$x\in B$. To show that such a neighborhood $N$ exists, let $\epsilon> 0$ be
such that $\{ x \, : \, \min_{a\in A} d(x,y) < \epsilon\} \subset U$. Then
there is a $K$ such that if $k\geq K$, then $\lim_{k \rightarrow\infty}
\vec{d}({\mathcal{F}}^{k}(x),A) < \epsilon/2$. By the continuity of the
functions in ${\mathcal{F}}$, there is a $\delta>0$ such that, if $d(x,y) <
\delta$, then $d(g(x), g(y) ) < \epsilon/2$ for all $g\in{\mathcal{F}}^{K}$.
Therefore ${\mathcal{F}}^{K}(N) \subset U$ and $\lim_{k \rightarrow\infty}
{\mathcal{F}}^{k}({\overline U^{\prime}}) = \lim_{k \rightarrow\infty}
{\mathcal{F}}^{k}({\overline U}) = A$.
\end{proof}

\section{Attractor Blocks}

\label{sec:block}

\begin{definition}
If $\mathcal{F}$ is an IFS on a compact metric space $\mathbb{X}$, then
$Q\subset\mathbb{X}$ is called an \textbf{attractor block} with respect to
$\mathcal{F}$ if $\mathcal{F}\left(  \overline{Q}\right)  \subset Q^{\circ}$.
\end{definition}

The following proposition is easy to verify.

\begin{prop}
\label{prop:block} If $Q$ is an attractor block with respect to the IFS
$\mathcal{F}$ on a compact metric space $\mathbb{X}$, then $\lim
_{k\rightarrow\infty}\mathcal{F}^{k}({\overline Q})= \bigcap_{k\rightarrow
\infty}\mathcal{F}^{k}({\overline Q})$ exists and is a Conley attractor of
$\mathcal{F}$.
\end{prop}

In light of Proposition~\ref{prop:block} we formulate the following definition.

\begin{definition}
If $Q$ is an attractor block and
\[
A = \bigcap_{k\rightarrow\infty}\mathcal{F}^{k}({\overline Q})
\]
is the corresponding Conley attractor, then $Q$ is called an \textbf{attractor
block for} $\mathbf{A}$ with respect to $\mathcal{F}$.
\end{definition}

The basin of a Conley attractor is not, in general, an attractor block. For,
if $B$ is the basin for a Conley attractor then, using the continuity of
$\mathcal{F}$, we have $\mathcal{F}\left(  \overline{B}\right)  =\overline{B}%
$. Therefore, unless $\overline{B}$ is open, $\mathcal{F}\left(  \overline
{B}\right)  $ is not contained in the interior of $B$, and so it cannot be an
attractor block. Nevertheless, the following theorem tells us that every
Conley attractor has a corresponding attractor block for it.

\begin{theorem}
\label{thm:block} If $\mathcal{F}$ is an IFS on a compact metric space, $A$ is
a Conley attractor of $\mathcal{F}$, and $\mathcal{N}$ is $^{{}}$a
neighborhood of $A$, then there is an attractor block for $A$ contained in
$\mathcal{N}$.
\end{theorem}

\begin{proof}
The proof will make use of the function $\mathcal{F}^{-1}(X) = \{ x
\in\mathbb{X }\, : \, f(x) \in X \; \text{for all}\; f \in{\mathcal{F}}\}$.
Note that ${\mathcal{F}}^{-1}$ takes open sets to open sets, $X \subset
({\mathcal{F}}^{-1}\circ{\mathcal{F}})(X)$ and $({\mathcal{F}}\circ
{\mathcal{F}}^{-1}) (X) \subset X$ for all $X$.

Let $U^{\prime}$ denote an open set containing $A$ such that $A =
\lim_{k\rightarrow\infty}\mathcal{F}^{k}(\overline{U^{\prime}})$, and let $U =
U^{\prime}\cap\mathcal{N}$. Let $V$ be an open set such that $A \subset V$ and
$\overline V \subset U$. Since $A = \lim_{k\rightarrow\infty}\mathcal{F}%
^{k}(\overline{V})$ by statement (2) of Proposition~\ref{prop:attractor},
there is an integer $m$ such that $\overline{{\mathcal{F}}^{k}(V)} \subset V$
for all $k > m$. Define $V_{k},\, k = 0,1,\dots, m,$ recursively, going
backwards from $V_{m}$ to $V_{0}$, as follows. Let $V_{m} = V$ and for $k =
m-1, \dots, 2,1,0,$ let $V_{k} = V \cap{\mathcal{F}}^{-1} (V_{k+1})$. If $O =
V_{0}$, then $O$ has the following properties:

\begin{enumerate}
\item $O$ is open,

\item $A \subset O$,

\item ${\mathcal{F}}^{k} (O) \subset V$ for all $k\geq0$.
\end{enumerate}

Property (2) follows from the fact that$A \subset{\mathcal{F}}^{-1}(A)$.
Property (3) follows from the facts that ${\mathcal{F}}^{k}(O) \subset V_{k}
\subset V$ for $0\leq k \leq m$, and ${\mathcal{F}}^{k}(O) \subset
{\mathcal{F}}^{k}(V) \subset V$ for all $k > m$.

Since, by statement (2) of Proposition~\ref{prop:attractor}, $A =
\lim_{k\rightarrow\infty}\mathcal{F}^{k}(\overline{O})$, there is an integer
$K$ such that ${\mathcal{F}}^{K}(\overline O) = \overline{{\mathcal{F}}%
^{K}(O)} \subset O$. Let $O_{k}, \, k = 0,1, \dots, K,$ be defined
recursively, going backwards from $O_{K}$ to $O_{0}$, as follows. Let $O_{K}$
be an open set such that ${\mathcal{F}}^{K}(\overline O) \subset O_{K} \subset
O$, and for $k=K-1, \dots, 2,1,0,$ let $O_{k}$ be an open set such that

\begin{enumerate}
\item ${\mathcal{F}}^{k} (\overline O) \subset O_{k} \subset U$, and

\item ${\mathcal{F}}(\overline O_{k}) \subset O_{k+1}$.
\end{enumerate}

To verify that a set $O_{k}$ with these properties exists, assume that $O_{k},
\, k \geq1,$ has been chosen with properties (1) and (2) and note that
${\mathcal{F}}^{k-1}(\overline O) \subset{\mathcal{F}}^{-1}({\mathcal{F}}%
^{k}(\overline O)) \subset{\mathcal{F}}^{-1}(O_{k})$ and ${\mathcal{F}}%
^{k-1}(\overline O) \subset\overline V \subset U$. Now choose $O_{k-1}$ to be
an open set such that ${\mathcal{F}}^{k-1}(\overline O) \subset O_{k-1}$ and
$\overline{O_{k-1}} \subset U \cap{\mathcal{F}}^{-1}(O_{k}) \subset U$. The
last inclusion implies ${\mathcal{F}}(\overline O_{k-1}) \subset O_{k}$.

We claim that
\[
Q = \bigcup_{k=0}^{K-1} O_{k}%
\]
is an attractor block for $A$. Since $A = {\mathcal{F}}^{k}(A) \subset
{\mathcal{F}}^{k} (O) \subset O_{k}$ for each $k$, we have $A \subset Q$.
Clearly $Q$ is an open set such that $Q \subset U \subset\mathcal{N}$. Hence
$A = \lim_{k\rightarrow\infty}\mathcal{F}^{k}(\overline{Q})$, and lastly,
\[
{\mathcal{F}}(\overline Q) = \bigcup_{k=0}^{K-1} {\mathcal{F}}(\overline
{O_{k}}) \subset\bigcup_{k=1}^{K} O_{k} = \bigcup_{k=1}^{K-1} O_{k} \cup O_{K}
\subset Q \cup O \subset Q \cup O_{0} \subset Q.
\]

\end{proof}

\section{Sufficient Conditions for a Conley Attractor to be a Stict Attractor}

\label{sec:SC}

In this section $A$ is a Conley attractor of a IFS ${\mathcal{F}}$ on a
compact metric space and $B$ is the basin of $A$. Under certain conditions $A$
is guaranteed to be a strict attractor. In particular, contractive properties
of the functions in ${\mathcal{F}}$ or of the \textquotedblleft fibers" of
${\mathcal{F}}$ may force this.

\begin{definition}
An IFS $\mathcal{F}$ on a metric space $(\mathbb{X},d)$ is said to be
\textbf{contractive} if there is a metric $\hat{d}$ inducing the same topology
on $\mathbb{X}$ as the metric $d$ with respect to which the functions in
$\mathcal{F}$ are strict contractions, i.e., there exists $\lambda\in
\lbrack0,1)$ such that $\hat{d}_{\mathbb{X}}(f(x),f(y))\leq\lambda\hat
{d}_{\mathbb{X}}(x,y)$ forall $x,y\in\mathbb{X}$ and for all $f\in\mathcal{F}$.
\end{definition}

A classical result of Hutchinson \cite{H}, a result marking the origin of the
concept of an iterated function system, states that if $\mathcal{F}$ is
contractive on a complete metric space $\mathbb{X}$, then $\mathcal{F}$ has a
unique strict attractor with basin $\mathbb{X}$. The corollary below follows
from Hutchinson's result.

\begin{corollary}
\label{cor:contractive} Let $A$ be a Conley attractor of an IFS ${\mathcal{F}%
}$ on a metric space and let $B$ be the basin of $A$. If $\mathcal{F}$ is
contractive on $B$, then $A$ is a strict attractor of ${\mathcal{F}}$ with
basin $B$.
\end{corollary}

\begin{proof}
If $S$ is any compact subset of $B$ containing $A$, then $S$ is a complete
metric space. Hutchinson's result implies that there is a unique strict
attractor $A^{\prime}$ in $B$ and that $B \subseteq B^{\prime}$, where
$B^{\prime}$ is the attractor of $A^{\prime}$. It only remains to show that
$A^{\prime}=A$ and $B^{\prime}=B$. Let $U$ be an open set containing $A$ and
$A^{\prime}$ and such that $\overline U\subset B$. Then by the definitions of
the Conley and strict attractor $A^{\prime}= \lim_{k\rightarrow\infty}
{\mathcal{F}}^{k}({\overline U} )= A$. Moreover, if $A \subset U
\subset{\overline U} \subset B^{\prime}$, then by the definition of strict
attractor we have $\lim_{k\rightarrow\infty} {\mathcal{F}}^{k}({\overline U} )
= A^{\prime}= A$. Therefore $B^{\prime}\subseteq B$.
\end{proof}

\begin{theorem}
\label{thm:affine} If ${\mathcal{F}}$ is an affine IFS or a M\"obius IFS with
a non-trivial Conley attractor $A$, then $A$ is a strict attractor and it is unique.
\end{theorem}

\begin{proof}
The theorem for the affine case follows from Theorem~\ref{thm:block},
\cite[Theorem 1.1]{ABVW}, and Corollary~\ref{cor:contractive}. According to
Theorem~\ref{thm:block} the Conley attractor $A$ has an attractor block $Q$.
According to \cite[Theorem 1.1]{ABVW}, if there is a compact set $Q$ such that
${\mathcal{F}}(Q) \subset Q^{o}$, then ${\mathcal{F}}$ is contractive on
${\mathbb{R}}^{n}$. By Corollary~\ref{cor:contractive} the Conley attractor
$A$ is a strict attractor. By Hutchinson's theorem, there is a unique strict
attractor in ${\mathbb{R}}^{n}$.

The proof in the M\"obius case is the same except that \cite[Theorem 1.1]{V}
is used in place of \cite[Theorem 1.1]{ABVW}.
\end{proof}

The analogous result to Theorem~\ref{thm:affine} fails for a projective IFS on
the projective plane $\mathbb{RP}^{2}$. See Example~\ref{ex:projNS} in
Section~\ref{sec:examples}.

The next theorem generalizes Corollary~\ref{cor:contractive} by replacing the
contractivity condition by a weaker condition called the
\textit{point-fibered} condition. Let $\Omega$ denote the set of all infinite
sequences $\{\sigma_{k}\}_{k=1}^{\infty}$ of symbols belonging to the alphabet
$\{1,...,N\}$. A typical element of $\Omega$ can be denoted as $\sigma
=\sigma_{1}\sigma_{2}\sigma_{3} \cdots$. With
\[
d_{\Omega}(\sigma,\omega) =
\begin{cases}
0 & \text{when $\sigma=\omega$}\\
2^{-k} & \text{when $k$ is the least index for which $\sigma_{k}\neq\omega
_{k}$},
\end{cases}
\]
$(\Omega, d_{\Omega})$ is a compact metric space called a \textit{code space}.
The topology on $\Omega$ induced by the metric $d_{\Omega}$ is the same as the
product topology that is obtained by treating $\Omega$ as the infinite product
space $\{1,...,N\}^{\infty}.$ For an IFS ${\mathcal{F}}$ and $\sigma\in\Omega
$, we use the shorthand notation
\[
f_{\sigma| k} = f_{\sigma_{1}}\circ f_{\sigma_{2}} \circ\cdots\circ
f_{\sigma_{k}}.
\]
The limit
\[
\pi_{{\mathcal{F}}}(\sigma,x) := \lim_{k\rightarrow\infty} f_{\sigma| k}(x),
\]
if it exists, is referred to as a \textit{fiber} of the IFS ${\mathcal{F}}$.

\begin{definition}
\label{def:ptF} An IFS $\mathcal{F}$ is \textbf{point-fibered} on a set $B$
if
\begin{equation}
\label{eq:ptF}\pi_{{\mathcal{F}}}(\sigma):=\lim_{k\rightarrow\infty}
f_{\sigma| k}(x),
\end{equation}
exists for all $\sigma\in\Omega$ and, for each $\sigma$, is independent of
$x\in B$.
\end{definition}

A main reason for the importance of the point-fibered concept is that it leads
to an addressing scheme for the points of a strict attractor, an addressing
scheme that relates naturally to the functions in the IFS. The addresses are
infinite strings in the alphabet $\{1,2,\dots, N\}$, where $N$ is the number
of functions in the IFS.  The precisely definition is as follows.

\begin{definition}
\label{def:coding} Let $\mathcal{F}$ be an IFS on a metric space $\mathbb{X}$
consisting of $N$ continuous functions. If $\pi:\Omega\rightarrow\mathbb{X}$
is a continuous mapping such that the following diagram commutes for all $n
\in\{1,2,\dots, N\}$
\begin{equation}%
\begin{array}
[c]{ccc}%
\Omega & \overset{s_{n}}{\rightarrow} & \Omega\\
\pi\downarrow\text{\ \ \ \ } &  & \text{ \ \ \ }\downarrow\pi\\
\mathbb{X} & \underset{f_{n}}{\rightarrow} & \mathbb{X}%
\end{array}
\label{commutediagram}%
\end{equation}
where $s_{n}$ is the inverse shift defined by $s_{n}(\sigma)= n\sigma$, then
$\pi$ is called a \textbf{coding map} for $\mathcal{F}$. The map $\pi$ is also
referred to as an \textbf{addressing function}.
\end{definition}

\begin{theorem}
\label{thm:ptF} If F is a point-fibered iterated function system on a compact
metric space $\mathbb{X}$, then

\begin{enumerate}
\item ${\mathcal{F}} : {\mathbb{H}}({\mathbb{X}}) \rightarrow{\mathbb{H}%
}({\mathbb{X}})$ has a unique fixed-point $A \in{\mathbb{H}}({\mathbb{X}})$,
i.e. ${{\mathcal{F}}}(A) = A$;

\item $A$ is the unique strict attractor of ${\mathcal{F}}$ in $\mathbb{X}$;

\item the basin of $A$ is $\mathbb{X}$;

\item the map $\pi_{{\mathcal{F}}} \, : \, \Omega\rightarrow\mathbb{X}$ given
by $\pi_{{\mathcal{F}}}(\sigma):=\lim_{k\rightarrow\infty} f_{\sigma| k}(x)$
is a coding map;

\item the range of the the coding map $\pi_{{\mathcal{F}}}$ is $A$, i.e.
$\pi(\Omega) = A$.
\end{enumerate}
\end{theorem}

\begin{proof}
The theorem follows from \cite[Proposition 4.4.2, p.107, Proposition 3.4.4,
p.77]{kieninger}.
\end{proof}

If ${\mathcal{F}}$ is an affine IFS, then the converse of statement (4) in
Theorem~\ref{thm:ptF} is true. A coding map of an affine IFS must be of the
form in Equation~\ref{eq:ptF}. The following result appears in \cite[Theorem
7.2]{ABVW}, where the \textit{affine hull} of a set is the smallest affine
subspace containing the set.

\begin{theorem}
If an affine IFS ${\mathcal{F}}$ on ${\mathbb{R}}^{n}$ has a coding map $\pi$
and the affine hull of $A := \pi(\Omega)$ equals ${\mathbb{R}}^{n}$, then
${\mathcal{F}}$ is point-fibered on ${\mathbb{R}}^{n}$. Moreover $A$ is the
strict attractor of ${\mathcal{F}}$.
\end{theorem}

The following generalization of Corollary~\ref{cor:contractive} follows from
Theorem~\ref{thm:ptF} in exactly the same way as
Corollary~\ref{cor:contractive} followed from Hutchinson's theorem. It is a
generalization because if it easy to show that if ${\mathcal{F}}$ is
contractive on a complete metric space $\mathbb{X}$, then ${\mathcal{F}}$ is
point-fibered on $\mathbb{X}$. Research on when the property of being
point-fibered implies that the the IFS is contractive is ongoing; see
\cite{kameyama}.

\begin{corollary}
Let ${\mathcal{F}}$ be an IFS on a compact metric space with a Conley
attractor $A$ and basin $B$. If ${\mathcal{F}}$ is point-fibered on $B$, then
$A$ is a strict attractor of ${\mathcal{F}}$ with basin $B$.
\end{corollary}

\section{Attractor-Repeller Pairs}

\label{sec:AR}

In this section it is assumed that the iterated function system is invertible.

\begin{definition}
A set $R\subset\mathbb{X}$ is said to be a \textbf{repeller} of the invertible
IFS $\mathcal{F}$ if $R$ is a Conley attractor of $\mathcal{F}^{\ast}%
=\{f^{-1}\,:\,f\in{\mathcal{F}}\}$. The \textbf{basin of a repeller} of
$\mathcal{F}$ is the basin for the corresponding Conley attractor of
$\mathcal{F}^{\ast}.$
\end{definition}

\begin{theorem}
\label{thm:ARpair} Let $\mathcal{F}$ be an invertible IFS on a compact metric
space $\mathbb{X}$. If $A$ is a Conley attractor of $\mathcal{F}$ with basin
$B$, then $A^{*} := \mathbb{X}\backslash B$ is a repeller of $\mathcal{F}$
with basin $\mathbb{X}\backslash A$.
\end{theorem}

\begin{proof}
By Theorem~\ref{thm:block} there is an attractor block $Q$ for $A$ with
respect to ${\mathcal{F}}$. It is easy to verify that the complement $Q^{*} :=
\mathbb{X}\setminus Q$ is an attractor block with respect to ${\mathcal{F}%
}^{*} $. Let $A^{*} = \lim_{k\rightarrow\infty}\mathcal{F^{*}}^{k}%
(\overline{Q^{*}})$ be the corresponding topological Conley attractor as
guaranteed by Proposition~\ref{prop:block}.

It is now sufficient to show that the basin $B$ of $A$ is $\mathbb{X
}\setminus A^{*}$, and to do this Lemma~\ref{lem:basin} is used. If $x \in Q$,
then $\lim_{k\rightarrow\infty} \vec{d} ({\mathcal{F}}^{k}(x),A) = 0$ because
$\lim_{k\rightarrow\infty} f^{k}({\overline Q}) = A$. Therefore $x \in B$. Now
let $x \in Q^{*}\setminus A^{*}$. If ${\mathcal{F}}^{k}(x) \subset Q$ for some
$k$, then again $\lim_{k\rightarrow\infty} \vec{d} ({\mathcal{F}}^{k}(x),A) =
0$ and the proof is complete. So, by way of contradiction, assume that
${\mathcal{F}}^{k}(x)$ is not a subset of $Q$ for any $k$. Then there is a set
$X = \{x_{k}\}$ such that $x_{k} \in{\mathcal{F}}^{k}(x)$ and $x_{k} \in
Q^{*}$. In this case $x \in{{\mathcal{F}}^{*}}^{k}(X) \subset{{\mathcal{F}%
}^{*}}^{k}(Q^{*})$ for all $k$. Since $\lim_{k\rightarrow\infty}%
\mathcal{F^{*}}^{k}(\overline{Q^{*}}) = A^{*}$, this implies that $x \in
A^{*}$, a contradiction.

It only remains to show that $B \cap A^{*} = \emptyset$. Let $x = x_{0}\in
A^{T}*$. Because ${\mathcal{F}}^{*}(A^{*}) = A^{*}$, by statement (1) of
Proposition~\ref{prop:attractor}, there is an $x_{1} \in A^{*}$ and an $f_{1}
\in{\mathcal{F}}$ such that $f_{1}^{-1}(x_{1}) = x_{0}$, i.e., $x_{1} =
f_{1}(x)$. For the same reason there is an $x_{2} \in A^{*}$ and an $f_{2}
\in{\mathcal{F}}$ such that $f_{2}^{-1}(x_{2}) = x_{1}$, i.e., $x_{2} = (f_{2}
\circ f_{1})(x)$. Continuing in this way, it is clear that ${\mathcal{F}}%
^{k}(x) \cap A^{*} \neq\emptyset$ for all $k>0$, which implies by
Lemma~\ref{lem:basin} that $x$ does not lie in $B$.
\end{proof}

\begin{definition}
If $\mathcal{F}$ is an invertible IFS on a compact metric space $\mathbb{X}$
and $A$ is a Conley attractor of $\mathcal{F}$ with basin $B$, then the set
\[
A^{\ast}:=\mathbb{X}\backslash B.
\]
is called the \textbf{dual repeller} of $A$.
\end{definition}

Examples of attractor-repeller pairs are shown in Section~\ref{sec:examples}.
The notion of a chain for an IFS is based on the notion of a chain for a
single function \cite{C1}.

\begin{definition}
Let $\varepsilon>0$ and let $\mathcal{F}$ be an IFS on $\mathbb{X}.$ An
$\varepsilon$\textbf{-chain} for $\mathcal{F}$ is a sequence of points
$\left\{  x_{i}\right\}  _{i=0}^{n}, \, n>0,$ in $\mathbb{X}$ such that for
each $i\in\{0,1,2,\dots,n-1\}$ there is an $f\in\mathcal{F}$ such that
$d(x_{i+1} ,f(x_{i}))<\varepsilon$. A point $x\in\mathbb{X}$ is
\textbf{chain-recurrent} for $\mathcal{F}$ if for every $\varepsilon>0$ there
is an $\varepsilon$-chain $\left\{  x_{i}\right\}  _{i=0}^{n}$ for
$\mathcal{F}$ such that $x_{0}=x_{n}=x$. The set of all chain recurrent points
for $\mathcal{F}$ is denoted by $\mathcal{R}:= \mathcal{R}\left(
\mathcal{F}\right) $.
\end{definition}

We refer to the following as the Conley-McGehee-Wiandt (CMW) theorem due to
previous versions in a non IFS context.

\begin{theorem}
\label{cmwthm} $[\text{CMW}]\;$ Let $\mathcal{F}$ be an invertible IFS on a
compact metric space $\mathbb{X}$. If $\mathcal{U}$ denotes the set of Conley
attractors of $\mathcal{F}$ and $\mathcal{R}$ denotes the set of chain
recurrent points of $\mathcal{F}$, then
\[
\mathcal{R}=\bigcap\limits_{A\in\mathcal{U}} ( A \cup A^{*}).
\]

\end{theorem}

\begin{proof}
First assume that $x \notin\bigcap\limits_{A\in\mathcal{U}} ( A \cup A^{*})$.
Then there is a Conley attractor $A$ such that $x \notin A \cup A^{*}$. By
Theorem~\ref{thm:ARpair} the point $x$ lies in the basin of $A$. According to
Theorem~\ref{thm:block} there is a closed attractor block $Q$ for $A$ with
respect to ${\mathcal{F}}$ such that $x\notin Q$. Since $x$ lies in the basin
of $A$, there is an integer $K$ such that ${\mathcal{F}} ^{k}(x) \subset
Q^{o}$ for all $k \geq K$. Since ${\mathcal{F}}(Q) \subset Q^{o}$,
\[
d := \min\, \{ d(y,y^{\prime}) \, : y\in{\mathcal{F}}(Q),\, y^{\prime}%
\in\partial Q\} > 0.
\]

We must show that $x \notin R$. By way of contradiction, assume that $x$ is
chain-recurrent and that $\{x_{i}\}_{i=0}^{n}$ an $\epsilon$-chain with $x_{0}
= x_{n} = x$ and $\epsilon< d$. We may assume, by repeating the chain if
necessary, that $n \geq K$. Since ${\mathcal{F}} ^{k}(x) \subset Q^{o}$ for
all $k\geq K$ and by the continuity of the functions in ${\mathcal{F}}$, if
$\epsilon$ is sufficiently small, say $\epsilon= \epsilon_{0}$, then $x_{K}
\in Q $. Now $d(x_{K+1}, f(x_{K})) <\epsilon_{0} <d$ for some $f\in
{\mathcal{F}}$. Therefore $f(x_{K}) \in Q^{o}$ implies that $x_{K+1} \in Q$.
Repeating this argument shows that $x_{i} \in Q$ for all $i \geq K$. From the
first paragraph in the proof $x\notin Q$ and by the above $x = x_{n} \in Q$, a contradiction.

Conversely assume that $x \notin\mathcal{R}$. Then there is an $\epsilon>0$
such that no $\epsilon$-chain starts and ends at $x$. Let $U$ denote the set
of all points $y$ such that there is an $\epsilon$-chain from $x$ to $y$.
Notice that (1) $x\notin U$, (2) $U$ is an open set, and (3) $\mathcal{F}%
({\overline U}) \subset U$. Therefore $A :=\lim_{k \rightarrow\infty
}{\mathcal{F}}^{k}({\overline U})$ is a Conley attractor with $x\notin A$.
Since ${\mathcal{F}}(x)\subset U$ and $A :=\lim_{k \rightarrow\infty
}{\mathcal{F}}^{k}({\overline U})$, the point $x$ lies in the basin of
$\mathcal{F}$, and therefore $x\notin A^{*}$ by Theorem~\ref{thm:ARpair}. So
$x\notin A \cup A^{*}$.
\end{proof}

\section{Examples}

\label{sec:examples}

\begin{example}
\label{ex:multiple} This is an example of an IFS with infinitely many Conley
attractors. Let $n$ be an integer and consider the IFS on ${\mathbb{R}}$
consisting of the single function
\[
f(x) =
\begin{cases}
x^{2} - 2nx + (n^{2}+n) & \; \text{if}\quad n \leq x < n+1, n\geq0\\
-x^{2} +2(n+1)x -(n^{2}+n) & \; \text{if}\quad n \leq x < n+1, n < 0.
\end{cases}
\]
For all integers $m,n \geq0$, the interval $[-m,n]$ is a Conley attractor with
basin $(-m-1,n+1)$.
\end{example}

\begin{example}
\label{ex:projNS} This example shows that the analougous result to
Theorem~\ref{thm:affine} fails for a projective IFS on the projective plane
$\mathbb{RP}^{2}$. Consider the IFS consisting of a single projective function
$f$ represented by the matrix
\[%
\begin{pmatrix}
2 & 0 & 0\\
0 & 1 & 0\\
0 & 0 & 1
\end{pmatrix}
.
\]
The line in $\mathbb{RP}^{2}$ corresponding to the $y,z-$plane in
$\mathbb{R}^{3}$ is a Conley attractor but it is not a strict attractor.
\end{example}

\begin{example}
\label{ex:ProjMob} $[\text{Contractive IFS}]$  Figure~\ref{Ex} shows the attractor-repeller pair of
a M\"obius IFS whose space is the Riemann sphere. Since this IFS is
contractive it has a unique non-trivial strict attractor-repeller pair. 
 For further details on such  M\"obius examples see \cite{V}.

\begin{figure}[htb] \label{fig:1}
\begin{center}
\vskip -5mm
\includegraphics[height=4in, width=3.6in] {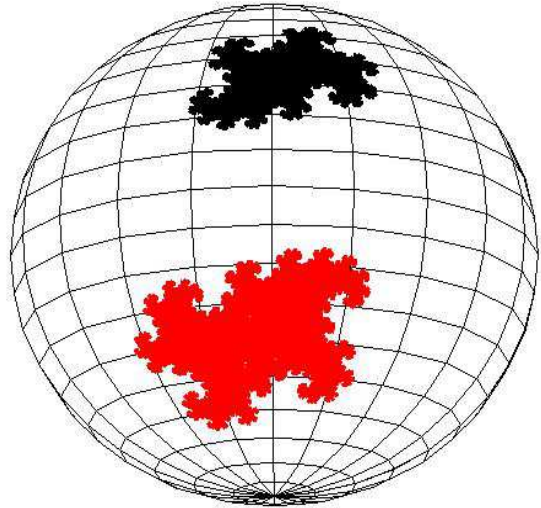}
\vskip -10mm
\caption{The attractor (red) and repeller (black) of a M\" obius IFS consisting of two M\"obius transformations.}
\label{Ex}
\end{center}
\end{figure}

\end{example}

\begin{example}
\label{ex:nonC} $[\text{Non-Contractive IFS}]$ A non-contractive IFS may have
no non-trivial Conley attractor. For example, the IFS on the unit circle
centered at the origin of the complex plane consisting of the single function
$f(z) = i z$ clearly has no Conley attractor. {\medskip}

Example~\ref{ex:multiple} is an non-contractive IFS with infinitely many
Conley attractors.{\medskip}

The following projective IFS, whose space is the projective plane, is
non-contractive but has a unique non-trivial strict attractor $A$ shown in
Figure~\ref{special}. The projective plane depicted as a disk with antipodal
points identified. This IFS consists of two functions given in matrix form by
\[
f_{1}=
\begin{pmatrix}
41 & -19 & 19\\
-19 & 41 & 19\\
19 & 19 & 41
\end{pmatrix}
\qquad\text{and}\qquad f_{2}=
\begin{pmatrix}
-10 & -1 & 19\\
-10 & 21 & 1\\
10 & 10 & 10
\end{pmatrix}
,
\]
The attractor $A$ is the union of the points in the red and green lines. In
the right panel a zoom is shown which displays the fractal structure of the
set of lines that comprise the attractor. The color red is used to indicate
the image of the attractor under $f_{1}$, while green indicates its image
under $f_{2}$. For further details on such examples, see \cite{BV}.{\medskip}

As indicated by the proof of Theorem~\ref{thm:affine}, an example of a
non-contractive affine or M\"obius IFS with a strict attractor cannot exist.

\begin{figure}[htb] \label{special}
\vskip 6mm
\centering
\includegraphics[natheight=13.652900in,natwidth=27.466900in,
height=2.0772in,
width=4.1486in]
{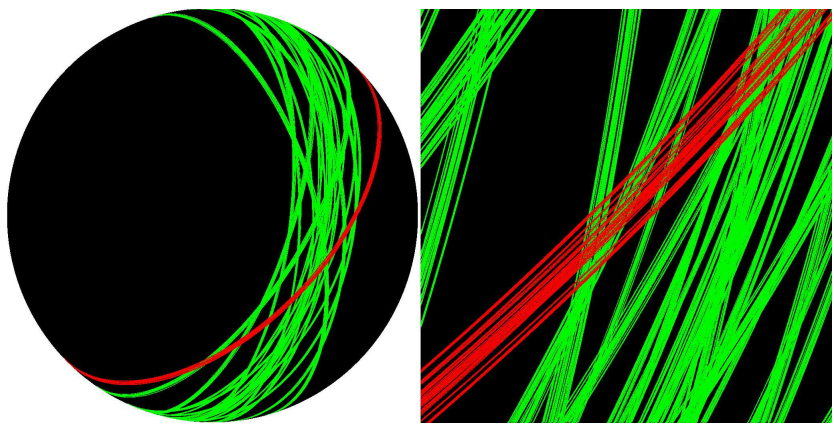}
\vskip 6mm
\caption{Projective attractor and a zoom.}
\end{figure}

\end{example}

\end{document}